\documentclass[11pt]{article}   
\newtheorem {Lemma} {Lemma}
\newtheorem {Theorem}  {Theorem}

\newenvironment {proof} {{\bf Proof.}}{\hspace*{\fill}$\Box$\par\vspace{4mm}}

\usepackage{amssymb,amsmath,wasysym}
\usepackage{latexsym}

\def\r{\rightarrow}

\def\p{\prime}
\def\:{\! :\!}

\def\lf{\lfloor}
\def\rf{\rfloor}

\newcommand{\calf}{\mathcal{F}}
\newcommand{\dlim}{\displaystyle\lim}
\newcommand{\calt}{\mathcal{T}}
\newcommand{\cals}{\mathcal{S}}
\newcommand{\calu}{\mathcal{U}}

\begin{document}
\baselineskip = 15pt
\bibliographystyle{plain}

\title{\bf Three Layer $Q_2$-Free Families in the Boolean Lattice}

\author{Jacob Manske\footnote{Corresponding author.}\\ Texas State University\\ San Marcos, TX, 78666\\ \texttt{jmanske@txstate.edu} \and  Jian Shen\\ Texas State University\\ San Marcos, TX 78666\\ \texttt{js48@txstate.edu} }

\maketitle

\begin{abstract}
 We prove that the largest $Q_2$-free family of subsets of $[n]$ which contains sets of at most three different sizes has at most $\left (3 + 2\sqrt {3} \right )N/3 + o(N)  \approx 2.1547N + o(N)$
members, where $N = { n \choose {\lf n/2 \rf }}$. This improves an earlier
bound of $2.207N + o(N)$ by Axenovich, Manske, and Martin.
\end{abstract}

\pagestyle{plain}

\section{Introduction and Motivation}

Let $Q_n$ be the $n$-dimensional Boolean lattice corresponding to subsets of an $n$-element set ordered by inclusion.
A poset $P = (X, \leq)$ is a subposet of $Q = (Y,\leq')$  if there is an injective map $f: X \rightarrow Y$ such that for $x_{1},x_{2} \in X$, $x_{1} \leq x_{2}$ implies $f(x_{1})\leq' f(x_{2})$.
For a poset $P$, we say that a set of elements $\calf \subseteq 2^{[n]}$ is $P$-free if   $(\calf, \subseteq)$  does not contain $P$ as a subposet.
Let $ex(n, P)$ be the size of the largest $P$-free family of subsets of $[n]$.
We say that  the set of all $i$-element subsets of $[n]$, $\binom{[n]}{i}$,  is the $i$th layer of $Q_n$.  Finally, let $N(n) = N = \binom{n}{\lfloor n/2 \rfloor }$;  i.e.,  $N$ is  the size of the largest layer of the Boolean lattice.

The first result in this area is Sperner's Theorem \cite{sperner28}, which states that $ex(n, Q_1) = N$. He also showed that the largest $Q_{1}$-free family is the largest layer in the Boolean lattice.

Many largest $P$-free families are simply unions of the largest layers in $Q_n$. For instance, the largest $Q_{1}$-free family is simply the largest layer in the Boolean lattice. In \cite{erdos45-sperner}, Erd\H{o}s generalized Sperner's result, showing that the size of the family of subsets of $[n]$ which does not contain
a chain with $k$ elements, $P_k$,  is equal to the number of elements in the $k-1$ largest layers of $Q_n$.   He also showed that the largest $P_{k}$-free family is the union of the $(k-1)$ largest layers in the Boolean lattice.

De Bonis, Katona and Swanepoel show in \cite{debkatswan05} that $ex(n,\Bowtie) = \binom{n}{\lfloor n/2\rfloor} + \binom{n}{\lfloor n/2\rfloor +1}$, where $\Bowtie$ is a subposet of $Q_n$ consisting of distinct sets $a, b, c, d$ such that
$a ,b \subset c, d$. They also showed that if $n=3$ or $n \geq 5$, the only $\Bowtie$-free family which achieves this size is the union of the two largest layers in the Boolean lattice. When $n=4$, there is another construction; take all subsets of size 2 together with $\{1\}, \{2\}, \{2,3,4\},$ and $\{1,3,4\}$.

 When an exact result is not known, often the asymptotic bounds for $ex(n,P)$ are expressed in terms of $N$.  De Bonis and Katona \cite{debkat2007} and independently Thanh \cite{Thanh98} showed that
$ex(n, V_{r+1}) = N +  o(N)$, where $V_{r+1}$ is a subposet of $Q_n$ with distinct elements $f, g_i$, $i=1,\ldots, r$,
 $f \subset g_{i}$ for $i = 1, \ldots, r$.  For a poset $K_{s,t}$,  with distinct elements $f_1, \ldots, f_s \subset g_1, \ldots, g_t$,  and
 a poset $P_k(s)$, with distinct elements $f_1 \subset \cdots \subset f_k \subset g_1, g_2, \ldots, g_s$,
  Katona and Tarjan  \cite{katonatarjan83}  and later De Bonis and Katona \cite{debkat2007} proved  that $ex(n, K_{s, t}) = 2N+o(N)$ and $ex(n, P_k(s)) = kN+o(N)$.
Griggs and Katona proved  in \cite{griggskatona08} that
$ex(n, \textsf{N}) = N + o(N)$, where $\textsf{N}$ is the poset with distinct elements $a,b,c,d$, such that  $a\subset c,d$,  and $b\subset c$.
Griggs and Lu \cite{griggslu09} proved that $ex(n,P_k(s,t)) = (k-1)N+o(N)$, where $P_k(s, t)$ is a poset with distinct elements
$f_1, f_2 \ldots, f_s \subset g_2\subset g_3 \subset  \cdots \subset  g_{k-1}\subset   h_1, \ldots, h_t$, $k\geq 3$.
 They also showed that $ex(n, O_{4k}) =N+o(N)$,   $ex(n, O_{4k-2})\leq (1+ \sqrt{2}/2) N+o(N)$, where $O_i$ is a poset of height two
 which is a cycle of length $i$ as an undirected graph. More generally, they proved that if $G=(V,E)$ is a graph and $P$ is a poset with elements $V\cup E$,
 with $v<e$ if $v\in V$, $e\in E$ and $v$ incident to $e$, then
 $ex(n,P) \leq \left(1+ \sqrt{1 - 1/(\chi(G)-1)}\right)N+o(N)$, where $\chi(G)$ is the
 chromatic number of $G$.
 Bukh \cite{bukh08} proved that $ex(n,T)= kN+o(N)$, where $T$ is a poset whose Hasse diagram is a tree and $k$ is the integer which is one less than the height of $T$. For a more complete survey on the subject, see \cite{griggslilu} and \cite{griggsli} for alternate proofs of some of the results listed above.

The smallest poset for which even an asymptotic result is not known is $Q_{2}$. In \cite{axenovichmanskemartin}, Axenovich, Martin, and the first author show that $ex(n,Q_{2}) \leq 2.283261N + o(N)$ and in the special case where  if $\calf$ is a family of subsets of $[n]$ with at most 3 different sizes and which is $Q_{2}$-free, then $|\calf| \leq 2.207N$. More recently, Griggs, Li, and Lu were able to show in \cite{griggslilu} that $\dlim_{n \to \infty}\dfrac{ex(n,Q_{2})}{N}\leq 2\frac{3}{11}$ (provided this limit exists), effectively showing that $ex(n,Q_{2}) \leq 2\frac{3}{11}N + o(N)$ and thus reducing the leading coefficient in the bound from \cite{axenovichmanskemartin} by about $.0105$. Our main result focuses on the special case where $\calf$ contains sets of at most 3 sizes; we state the result below as Theorem \ref{maintheorem}.

\begin{Theorem} Let $n$ be a positive integer.
If $\calf \subset Q_{n}$ is a  $Q_{2}$-free family, $\calf= \cals\cup \calt \cup \calu$,   where $\cals$ is a collection of minimal elements of $\calf$,
$\calu$ is a collection of maximal elements of $\calf$ and $\calt = \calf\setminus (\cals\cup \calu) $ such that for any $T\in \calt$, $S\in \cals$, $U\in \calu$,  $|T| =k$, $|U|>k$, $|S|<k$,
then $|\calf| \leq  \left (3 + 2\sqrt {3} \right )N/3 + o(N) \approx 2.1547N + o(N)$.
In particular, if $\calf$ is a $Q_2$-free  subset of three layers of $Q_n$,  then  $|\calf| \leq 2.1547N+o(N)$.
\label{maintheorem}
\end{Theorem}

\section{Proof of Theorem \ref{maintheorem}}

Following the argument in \cite{axenovichmanskemartin}, it suffices to prove Theorem \ref{maintheorem} in the case where $\calf$ contains sets of size $k$, $(k-1)$, and $(k+1)$.

For two functions $A(n)$ and $B(n)$, by $A(n) \lesssim B(n)$ (or $B(n) \gtrsim A(n)$) we mean
$$\mbox{either } A(n) \le B(n) \mbox{ or } \lim_{ n \r \infty} \frac {A(n)} {B(n)} = 1.$$
Suppose ${\cal F}$ is a $Q_2$-free family from three layers, $L_1$, $L_2$, $L_3$,  of the Boolean
lattice $Q_n$, where $L_1 = \binom{[n]}{k-1}$,   $L_2 = \binom{[n]}{k}$, $L_3= \binom{[n]}{k+1}$, and by Lemma 1 from \cite{axenovichmanskemartin}, we may assume $n/2 - n ^{2/3} \leq k \leq n/2 + n ^{2/3}$. Let $\cals = \calf \cap L_1$, $\calt = \calf \cap L_2$, $\calu = \calf \cap L_3. $

 For $X \in L_1$, $Y \in L_2$, and $Z \in L_3$, we define
$$\begin{array}{ll}
f(X) = |\{ T \in \calt: X \subset T\}|; & g(Z) = |\{ T \in \calt: Z \supset T\}|;\\
\breve f( Y) = |\{ S \in \cals: S \subset Y\}|; & \breve g(Y) = |\{ U \in \calu: U \supset Y\}|;
\end{array}$$
Note that
$$\sum_{X \in \cals} f(X) = \sum _{Y \in \calt} \breve f( Y)
\mbox{ and }
\sum_{Z \in \calu} g(Z) = \sum _{Y \in \calt} \breve g( Y).$$
From  \cite{axenovichmanskemartin}, we have
\begin{equation} \label{eq:no1}
{\cal |U|+|T|+|S| }\le 2N -|\calt| + \frac 1k \sum _{Y \in \calt} \left ( \breve f (Y) + \breve g (Y) \right ),
\end{equation}
where $N = \binom{n}{\lf n/2 \rf } \approx {n \choose k}$.
We start with  a few lemmas involving some counting arguments.

\begin{Lemma} \label{Lemma 2}
For any $X \in \cals$ and $Y \in \calt$ with $X \subset Y$,
$$f(X) + \breve g(Y) \le n-k+1 \lesssim k.$$
For any $Y \in \calt$ and $Z \in \calu$ with $Y \subset Z$,
$$g(Z) + \breve f(Y) \le k+1 \lesssim k.$$
\end{Lemma}

\begin{proof} We only prove $f(X) + \breve g(Y) \le n-k+1 \lesssim k$, and the other inequality
follows similarly.
By definition, $\breve g(Y) = |\{ U \in \calu: U \supset Y\}| \le n -|Y| = n -k$.
So we may suppose without loss of generality that $f(X) \geq 2$.
For any $Y' \in \calt $ with $X \subset Y' \ne Y$, we have $|Y\cup Y'| = |Y| + |Y'| - |Y \cap Y^\p|
= |Y| + |Y'| - |X| = 2k -(k-1) = k+1$ and thus $Y\cup Y' \in L_3$.
Since ${\cal F}$ is $Q_2$-free,
we have $Y\cup Y' \in \left ( L_3 - \calu \right )$.
Further $Y\cup Y' \ne Y\cup Y''$ for any other $Y''$ with $X \subset Y'' \in
\left ( \calt - \{Y, Y' \} \right ). $
Therefore,
\begin{eqnarray*}
\breve g(Y) & \le & |\{ U \in \calu: U \supset Y \mbox{ and } U \ne Y\cup Y' \mbox{ with }
X \subset Y' \in \left ( \calt - \{Y\} \right ) \}| \vspace{.2cm} \\
& = & |\{ U \in \calu: U \supset Y\}|- |\{ Y' \in \calt : X \subset Y'
 \in \left ( \calt - \{Y\} \right ) \}|\vspace{.2cm} \\
& = & n-k - f(X) +1.
\end{eqnarray*}
\end{proof}

\begin{Lemma} \label{Lemma 3}
$$|\cals| \gtrsim \sum _{Y \in \calt}  \frac {\breve f (Y)}{k - \breve g (Y)},$$
 $$|\calu| \gtrsim \sum _{Y \in \calt}  \frac {\breve g (Y)}{k - \breve f (Y)}.$$
\end{Lemma}

\begin{proof} We use double counting to only prove the first inequality, and the other inequality
follows symmetrically. First
$$\sum_{ (X, Y):\  \cals \ni X \subset Y \in \calt } \frac 1 {f(X)}
= \sum _{X \in \cals} \ \  \sum _{X \subset Y \in \calt } \frac 1 {f(X)} = \sum _{X \in \cals}1 =|\cals|.$$
Second, by Lemma~\ref{Lemma 2},
$$\sum_{ (X, Y): \ \cals \ni X \subset Y \in \calt } \frac 1 {f(X)}
=\sum _{Y \in \calt} \ \  \sum _{\cals \ni X \subset Y } \frac 1 {f(X)}
\gtrsim \sum _{Y \in \calt} \ \  \sum _{\cals \ni X \subset Y } \frac 1 {k- \breve g (Y)} =
\sum _{Y \in \calt}  \frac {\breve f (Y)}{k - \breve g (Y)}.$$
\end{proof}

\begin{Lemma} \label{Lemma 5}
For any non-negative reals $x$ and $y$ with $x < k$, $y < k$, and $x+y \ge k$,
$$\frac x {k-y} + \frac y {k-x} \ge \frac {2x+2y}{2k -x-y}.$$
\end{Lemma}

\begin{proof}
$$x(k-x)(2k-x-y) + y(k-y)(2k-x-y)- (2x+2y)(k-x)(k-y) = (x-y)^2(x+y-k) \ge 0.$$
\end{proof}

Define $\calt_1 := \{ Y \in \calt: \breve f (Y) + \breve g (Y) \ge k \}$
and $\calt_2 := \{ Y \in \calt: \breve f (Y) + \breve g (Y) < k \}$.

\begin{Lemma} \label{modified Lemma 7}
$$\sum _{Y \in \calt_2} \left ( \breve f (Y) + \breve g (Y) \right )
\le k |\calt_2|,$$
$$\sum _{Y \in \calt_1} \left ( \breve f (Y) + \breve g (Y) \right )
\lesssim \frac {2k |\calt_1| \left (  |\cals|  +  |\calu| \right )}
{|\cals|  +  |\calu| + 2|\calt_1| }.$$
\end{Lemma}

\begin{proof} By the definition of $\calt_2$, we have $\sum _{Y \in \calt_2} \left ( \breve f (Y) + \breve g (Y) \right ) \le \sum _{Y \in \calt_2} k = k |\calt_2|.$ We now prove the second inequality of the lemma. Recall that $\calt_1 \cup
\calt_2$ forms a disjoint union of $ \calt$. By Lemma~\ref{Lemma 3} and Lemma~\ref{Lemma 5}
(with $x=\breve f (Y)$ and $y=\breve g (Y)$),
$$ \begin{array}{ll}
|\cals|  +  |\calu| & \gtrsim
\displaystyle \sum _{Y \in \calt}  \left ( \frac {\breve f (Y)}{k - \breve g (Y)} + \frac {\breve g (Y)}{k - \breve f (Y)} \right ) \vspace{.2cm} \\
& \ge \displaystyle \sum _{Y \in \calt_1} \frac {2\breve f (Y) + 2\breve g (Y)} { 2k - \breve f (Y) - \breve g (Y)} \vspace{.2cm}\\
& = \displaystyle \sum _{Y \in \calt_1} \left ( \frac {4k} { 2k - \breve f (Y) - \breve g (Y)} -2 \right )
 \vspace{.2cm}\\
& = \displaystyle \sum _{Y \in \calt_1} \frac {4k} { 2k - \breve f (Y) - \breve g (Y)} -2|\calt_1|.
\end{array} $$
Then, by the Cauchy-Schwarz inequality,
$$\begin{array}{ll}
& \left ( |\cals|  +  |\calu| + 2|\calt_1| \right )
\left ( 2k |\calt_1|  - \sum _{Y \in \calt_1} \left (\breve f (Y) + \breve g (Y)\right ) \right ) \vspace{.2cm} \\
\gtrsim & \displaystyle \sum _{Y \in \calt_1} \frac {4k} { 2k - \breve f (Y) - \breve g (Y)}
\sum _{Y \in \calt_1} \left ( 2k - \breve f (Y) - \breve g (Y) \right ) \ge 4k |\calt_1| ^2,
\end{array} $$
which is equivalent to
$$\sum _{Y \in \calt_1} \left ( \breve f (Y) + \breve g (Y) \right )
\lesssim 2k|\calt_1|- \frac {4k |\calt_1|^2 } { |\cals|+|\calu|+2|\calt_1|}
=\frac {2k |\calt_1| \left (  |\cals|  +  |\calu| \right )}
{|\cals|  +  |\calu| + 2|\calt_1| }.$$
\end{proof}

We are now ready to prove Theorem \ref{maintheorem}.

\noindent
\begin{proof} We will show that $$|\cals|  +  |\calu| + |\calt| \lesssim
\frac {3 + 2\sqrt {3}} {3}N \approx 2.1547N.$$
By (\ref{eq:no1}) and Lemma \ref{modified Lemma 7},
\begin{eqnarray*}
{\cal |U|+|T|+|S| }
& \lesssim & 2N -|\calt| + \frac 1k \sum _{Y \in \calt} \left ( \breve f (Y) + \breve g (Y) \right ) \vspace{.2cm}\\
& =  & 2N -|\calt_1|-|\calt_2|   + \frac 1k \sum _{Y \in \calt_1} \left ( \breve f (Y) + \breve g (Y) \right ) + \frac 1k \sum _{Y \in \calt_2} \left ( \breve f (Y) + \breve g (Y) \right )  \vspace{.2cm}\\
& \lesssim & 2N -|\calt_1| + \frac {2 |\calt_1| \left (  |\cals|  +  |\calu| \right )}
{|\cals|  +  |\calu| + 2|\calt_1| }\vspace{.2cm}\\
& = &  2N + \frac {|\calt_1| \left (  |\cals|  +  |\calu| - 2|\calt_1|  \right )}
{|\cals|  +  |\calu| + 2|\calt_1| }\vspace{.2cm}\\
& \lesssim & 2N + \left ( |\calu|+|\calt_1|+|\cals|  \right )\cdot f(x),
\end{eqnarray*}

where $$f(x) = \frac {x-2}{(x+2)(x+1)}$$
with $x= \dfrac{|\cals|  +  |\calu|}{|\calt_1|} > 0$. Note that the function $f(x)$ has a unique critical point at $x= \left (2 + 2\sqrt {3} \right )$ on the interval $[0, \infty)$. The function $f(x)$ achieves its maximum value at $x= \left (2 + 2\sqrt {3} \right )$
and thus
$${\cal |U|+|T|+|S| } \lesssim 2N+ \left ( {\cal |U|+|T|+|S| } \right ) \cdot f\left (2+ 2 \sqrt 3\right )
=2N+ \left (7- 4 \sqrt 3 \right ) \left ( {\cal |U|+|T|+|S| } \right ),$$
from which we have ${\cal |U|+|T|+|S| }\lesssim  \left (3 + 2\sqrt {3} \right )N/3  \approx 2.1547N $.
\end{proof}

\section{Future work}

There are two ways of extending the argument to get a general bound on $ex(n,Q_{2})$. One is to adapt the counting argument above to work with a family of subsets of $[n]$ with more than 3 sizes. Another way is to show that if $\calf$ is a family of size $ex(n,Q_{2})$, then $\calf$ contains sets of at most 3 different sizes. If the latter is true, then Theorem 1 shows that $ex(n,Q_{2}) \leq 2.1547N$.

We may also investigate $\pi(Q_{2}) = \dlim_{n\to\infty}\dfrac{ex(n,Q_{2})}{N}$, as the authors in \cite{griggslilu} do. It is not known if $\pi(Q_{2})$ exists, although it is conjectured in \cite{griggslu09} that $\pi(P)$ exists and is an integer for any finite poset $P$. If true, then $\pi(Q_{2}) = 2$ and $ex(n,Q_{2}) = 2N + o(N)$.

\noindent
{\bf Acknowledgments.}
Shen's research was partially supported by NSF (CNS 0835834, DMS 1005206) and Texas Higher Education
 Coordinating Board (ARP 003615-0039-2007).


\bibliography{/Users/jmanske/Documents/Studium/iastate/research/myrefs}

\end{document}